\documentclass[final,onefignum,onetabnum]{siamart220329}
\usepackage{braket,amsfonts}

\usepackage{amsmath}
\usepackage{amssymb}
\usepackage{graphicx}
\usepackage[caption=false]{subfig}
\usepackage{bm}
\usepackage{epstopdf}
\usepackage{xcolor}
\usepackage{multirow}

\newcommand{\Div}{\operatorname{div}}
\newcommand{\varepsilonh}{\varepsilon_h}
\newcommand{\HatDiv}{\hat{\operatorname{div}}}
\newcommand{\HatNab}{\hat{\nabla}}

\newcommand{\hatv}{\hat{v}}

\newcommand{\hattau}{\hat{\tau}}
\newcommand{\bmx}{\mathbf{x}}
\newcommand{\bbT}{\mathbb{T}}
\newcommand{\bbS}{\mathbb{S}}
\newcommand{\bbR}{\mathbb{R}}
\newcommand{\bbX}{\mathbb{X}}
\newcommand{\bbRn}{\mathbb{R}^2}
\newcommand{\bbRnn}{\mathbb{R}^{2\times 2}}

\newcommand{\bbAm}{\mathbb{A}^m}
\newcommand{\maA}{\mathcal{A}}
\newcommand{\maO}{\mathcal{O}}
\newcommand{\Tr}{\operatorname{tr}}
\newcommand{\bdI}{I}
\newcommand{\idI}{\mathrm{I}}
\newcommand{\hatsigma}{\hat{\sigma}}
\newcommand{\hatu}{\hat{u}}
\newcommand{\hatsigmah}{\hat{\sigma}_h}
\newcommand{\hatuh}{\hat{u}_h}
\newcommand{\hattauh}{\hat{\tau}_h}
\newcommand{\hatvh}{\hat{v}_h}
\newcommand{\hatphih}{\hat{\phi}_h}

\newcommand{\Vhsm}{V_h^{m,*}}

\newcommand{\Thm}{\mathcal{T}_h^m}
\newcommand{\Thl}{\mathcal{T}_h^l}
\newcommand{\Tho}{\mathcal{T}_h^1}
\newcommand{\Th}{\mathcal{T}_h}

\newcommand{\TPhm}{\mathcal{T}_{\partial,h}^m}

\newcommand{\TPho}{\mathcal{T}_{\partial,h}^1}
\newcommand{\TIho}{\mathcal{T}_{0,h}^1}
\newcommand{\TIhm}{\mathcal{T}_{0,h}^m}
\newcommand{\Zhm}{\mathcal{Z}_h^m}
\newcommand{\Km}{K^m}
\newcommand{\Kl}{K^l}

\newcommand{\Ko}{K^1}
\newcommand{\El}{E^l}
\newcommand{\Em}{E^m}
\newcommand{\Eh}{\mathcal{E}_h}
\newcommand{\Ehm}{\mathcal{E}_h^m}
\newcommand{\EPhm}{\mathcal{E}_{\partial,h}^m}
\newcommand{\EPhl}{\mathcal{E}_{\partial,h}^l}
\newcommand{\EIhm}{\mathcal{E}_{0,h}^m}

\newcommand{\EPh}{\mathcal{E}_{\partial,h}}
\newcommand{\EIh}{\mathcal{E}_{0,h}}

\newcommand{\Fhm}{\mathcal{E}^m_{h}}
\newcommand{\FKm}{F_K^m}
\newcommand{\FKl}{F_K^l}
\newcommand{\FKo}{F_K^1}
\newcommand{\Fm}{F^m}
\newcommand{\Philm}{\Phi^{lm}}

\newcommand{\PhiKlm}{\Phi_K^{lm}}
\newcommand{\PhiKom}{\Phi_K^{1m}}
\newcommand{\PsiKm}{\Psi_K^m}
\newcommand{\Psim}{\Psi^m}
\newcommand{\Qh}{Q_h}
\newcommand{\Pih}{\Pi_h}
\newcommand{\Piho}{\Pi_h^1}

\newcommand{\Sigh}{\Sigma_h}
\newcommand{\Vh}{V_h}
\newcommand{\Sighm}{\Sigma_h^m}
\newcommand{\Vhm}{V_h^m}
\newcommand{\Sigho}{\Sigma_h^1}
\newcommand{\Vho}{V_h^1}
\newcommand{\vh}{v_h}
\newcommand{\sigmah}{\sigma_h}
\newcommand{\tauh}{\tau_h}
\newcommand{\uh}{u_h}

\newcommand{\hatush}{\hat{u}^*_h}

\newcommand{\Omo}{{\Omega^1}}
\newcommand{\Omm}{{\Omega^m}}
\newcommand{\Oml}{{\Omega^l}}
\newcommand{\OmmS}{{\Omega^m_S}}
\newcommand{\OmlS}{{\Omega^l_S}}

\newcommand{\aam}{{a^m}}
\newcommand{\bbm}{{b^m}}
\newcommand{\bbo}{{b^1}}

\newsiamremark{remark}{Remark}
\newsiamremark{Hy}{Hypothesis}

\usepackage[most]{tcolorbox}
\begin{tcbverbatimwrite}{tmp_\jobname_header.tex}
\title{The Hu-Zhang element for linear elasticity on curved domains\thanks{Submitted to the editors August 28, 2025. \funding{The third author was supported by the National Natural Science Foundation of China Grants No. NSFC 12288101.}}}

\author{Wei Chen\thanks{LMAM and School of Mathematical Sciences, Peking University, Beijing 100871, P. R. China~\&~Chongqing Research Institute of Big Data, Peking University, Chongqing 401329, P. R. China (\email{2406397052@pku.edu.cn}).}
\and Xinyuan Du\thanks{LMAM and School of Mathematical Sciences, Peking University, Beijing 100871, P. R. China (\email{2301110045@stu.pku.edu.cn},
\email{hujun@math.pku.edu.cn}).}
\and Jun Hu\footnotemark[3]
}

\headers{The Hu-Zhang element for linear elasticity on curved domains}{Wei Chen, Xinyuan Du, and Jun Hu}
\end{tcbverbatimwrite}
\input{tmp_\jobname_header.tex}

\ifpdf
\hypersetup{ pdftitle={Guide to Using  SIAM'S \LaTeX\ Style} }
\fi

\begin{document}
\maketitle

\begin{tcbverbatimwrite}{tmp_\jobname_abstract.tex}
\begin{abstract}
This paper extends the Hu-Zhang element for linear elasticity to curved domains, preserving strong symmetry and $H(\Div)$-conformity.
The non-polynomial structure of the curved Hu-Zhang element makes it difficult to analyze the stability, which is overcome by establishing a novel inf-sup condition.
Optimal convergence rates are achieved for all variables except for the stress in the $L^2$-norm.
This suboptimality originates from the fact that the divergence space of the curved Hu-Zhang element is not contained in the discrete displacement space, which is improved by local~$p$-enrichment on boundary elements. 
Two numerical experiments validate the theoretical results.
\end{abstract}


\end{tcbverbatimwrite}
\input{tmp_\jobname_abstract.tex}

\section{Introduction}
\label{intro}
Let $\Omega\subset\bbRn$ be a bounded connected domain with piecewise~$C^{r+1}$ boundary~$\partial\Omega$. Given a load~$f\in L^2(\Omega;\bbRn)$, consider the Hellinger-Reissner formulation for the linear elasticity problem with homogeneous displacement boundary condition: find $(\sigma, u)\in\Sigma \times V:=H(\Div,\Omega; \bbS)\times L^2(\Omega; \bbRn)$, such that
\begin{equation}\label{eq1}
\begin{cases}
a(\sigma,\tau) + b(\tau,u)=0, &\text{for all}~\tau\in\Sigma,\\
b(\sigma,v)=-(f,v),&\text{for all}~ v\in V,
\end{cases}
\end{equation}
where the bilinear forms are defined by
\begin{equation}\label{abdefine}
a(\sigma,\tau):=(\maA\sigma,\tau),\quad
b(\tau,v):=(\Div\tau,v),
\end{equation}
and $\bbS$ denotes the space of symmetric matrices of~$\bbRnn$.
The compliance tensor~$\maA=\maA(x):\bbS\to\bbS$,
characterizing material elastic properties, is bounded and symmetric positive definite uniformly for $x\in\Omega$.
For isotropic materials,~$\maA$ takes the form:
\begin{equation*}
\maA\sigma:=\frac{1}{2\mu}\left(\sigma-\frac{\lambda}{2\lambda+2\mu}(\Tr\sigma)\bdI\right).
\end{equation*}
Here, $\bdI$ is the $2\times 2$ identity matrix and $\Tr(\cdot)$ denotes the trace operator.
The Lam\'{e} constants $\lambda$ and~$\mu$ are positive material parameters.

Constructing stable mixed finite element spaces for~\eqref{eq1} is a long-standing and challenging problem~\cite{arnold2002mixed,arnold2003nonconforming}.
Early developments in this direction include composite elements~\cite{arnold1984family,1978Some,Watwood1968An} and weakly symmetric approaches~\cite{1979Equilibrium,1984PEERS,1989A,1988A}. In \cite{arnold2002mixed}, Arnold and Winther designed the first family of mixed finite element methods in two dimensions, based on polynomial shape function spaces.
From then on, various stable mixed elements have been constructed, see, e.g., \cite{2005A,arnold2005rectangular,arnold2008finite,arnold2007mixed,arnold2002mixed,arnold2003nonconforming,awanou2012two,boffi2009reduced,chen2011conforming,cockburn2010new,gopalakrishnan2011symmetric,gopalakrishnan2012second,guzman2010unified,hu2014simple,hu2008lower,man2009lower,yi2005nonconforming,yi2006new}.
Recently, Hu \cite{hu2015finite} proposed a family of conforming mixed elements on simplicial meshes in any dimension, with specific implementations in 2D \cite{hu2014family} and 3D \cite{hu2015family}.
This new class of elements requires much fewer degrees of freedom than those in the earlier literature. 
For $k\geq n+1$, the stress is discretized by piecewise~$P_k$ polynomials with $H(\Div)$-conformity and the displacement by piecewise $P_{k-1}$ polynomials.
Moreover, a new idea is proposed to analyze the discrete inf-sup condition, and the basis functions therein are easy to obtain. 
For the case with~$1\leq k\leq n$, the symmetric matrix-valued finite element space is enriched by proper high-order $H(\Div)$ bubble functions to stabilize the discretization in~\cite{hu2016finite}. Another method by stabilization techniques to deal with this case can be found in \cite{chen2017stabilized,chen2018fast,verfurth1986finite}. Corresponding mixed elements on both rectangular and cuboid meshes are constructed in \cite{hu2015new}.

In practical engineering analyses of curved domains, the consistency error induced by polygonal boundary approximations fundamentally limits the convergence rates of finite element methods.
For second-order elliptic equations discretized with $k$-th order Lagrange finite element methods on meshes of geometric approximation order~$m$, the $H^1$-error between the numerical solution and the exact solution is dominated by two key parts: the error of the geometric approximation~$\maO(h^m)$ and the error of the polynomial approximation~$\maO(h^k)$. 
Rigorous proofs for this error estimate can be found in \cite{bernardi1989optimal,ciarlet1972interpolation,lenoir1986optimal,scott1973finite}.
To guarantee optimal convergence rates, $m$ must equal or exceed~$k$~(isoparametric or superparametric approximation).
Although curved Lagrange finite element methods are mature \cite{Brenner2013Isoparametric,brenner2008mathematical,Caubet2024APriori,Song2021Second}, the application of mixed formulation to linear elasticity on curved domains remains scarce.
By imposing Piola's transformation \cite[(2.1.69)]{boffi2013mixed}, Qiu and Demkowicz \cite{Qiu2011Mixed} generalized the weakly symmetric Arnold-Falk-Winther element \cite{arnold2007mixed} to curved domains.
Through extending polytopal template methods \cite{sky2024polytopal,sky2024higher}, Sky, Neunteufel, Hale and Zilian \cite{Sky2023AReissner} developed a novel fourth-order tensor-based polytopal transformation mapping Hu-Zhang element basis functions from reference to physical elements.
Such approach is also applicable to curved elements.
For other problems with mixed finite element methods, we refer readers to~\cite{arnold2020hellan,bae2004finite,Bertrand2014First,Bertrand2016Parametric,dione2015penalty,neilan2021divergence,verfurth1986finite,Walker2022The} and the references therein.

This paper generalizes the Hu-Zhang element to curved triangulations while preserving strong symmetry and $H(\Div)$-conformity.
The loss of the polynomial nature of the curved Hu-Zhang element introduces significant challenges in analyzing the discrete inf-sup condition and establishing error estimates between the numerical solution and the exact solution.
First, this non-polynomial structure makes it difficult to characterize the divergence space of $H(\Div)$ bubble function space, which is a critical step in proving the discrete inf-sup condition on simplicial meshes as established in~\cite{hu2015finite,hu2014family,hu2015family}.
Second, the deviation from polynomials implies that the divergence space of the curved Hu-Zhang element is not contained in the discrete displacement space. 
This severely complicates the error estimate of stress in the~$L^2$-norm through standard arguments (such as those for simplicial meshes in~\cite[Remark 3.1]{hu2015finite}),
particularly when the error of the geometric approximation coexists.

To overcome the first difficulty, we establish a novel discrete inf-sup condition on triangular meshes by substituting the $H(\Div)$-norm of discrete stress with a broken~$H^1$-norm.
This serves as a key step to recasting $H(\Div)$-stability analysis on curved meshes as broken $H^1$-stability analysis on triangular meshes.
The second difficulty is resolved by extending the stability in mesh-dependent norms for simplicial meshes, originally established in \cite{chen2018fast}, to curved meshes.
Moreover, superconvergence results for the displacement follow directly from this extended stability.
It is noteworthy that the failure of the inclusion~(i.e., $\Div\Sighm\not\subset\Vhm$) introduces an error term in analyzing both $L^2$-error of stress and superconvergence of displacement, leading to a half-order reduction from the optimal convergence rate.
Crucially, this degradation is mitigated by locally increasing the polynomial degree of the curved Hu-Zhang element. This degree enhancement applies exclusively to boundary elements, only leading to a marginal increase in the total computation cost.

The rest of this paper is organized as follows. Section \ref{notation} introduces
some notations.
Section \ref{tirhz} briefly reviews the Hu-Zhang element on triangular meshes, and Section \ref{CTR} provides a quick review of curved triangulations. 
The extension of the Hu-Zhang element to curved meshes and the proof of the discrete inf-sup condition are presented in Section~\ref{curvehz}, followed by error estimates in Section~\ref{erroranalysis}. Finally, Section \ref{numericres} provides two numerical tests validating the theoretical analysis.

\subsection{Notations}\label{notation}
For a bounded domain $G\subset\bbRn$ and a nonnegative integer $m$, denote the Sobolev space
\begin{equation*}
H^m(G;\bbX):=\{\phi:G\to\bbX~:~D^\alpha\phi\in L^2(G;\bbX),~|\alpha|\leq m\},
\end{equation*}
where $\bbX$ is a finite-dimensional vector space ($\bbX=\bbR$, $\bbRn$ or $\bbS$).
The associated norm and semi-norm are denoted by $\|\cdot\|_{m,G}$ and $|\cdot|_{m,G}$, respectively.
Further, for a non-integer~$s$, let $H^s(G;\bbX)$ be the intermediate space with norm $\|\cdot\|_{s,G}$ by interpolating Sobolev spaces of integer order. 
The $L^2$-inner product on $G$ is written as $(\cdot,\cdot)_G$. When $G=\Omega$, we simplify notations:~$\|\cdot\|_m:=\|\cdot\|_{m,\Omega}$, $|\cdot|_m:=|\cdot|_{m,\Omega}$, $\|\cdot\|_s:=\|\cdot\|_{s,\Omega}$ and $(\cdot,\cdot):=(\cdot,\cdot)_\Omega$. Let $H^m_0(G;\bbRn)$ be the closure of $C^\infty_0(G;\bbRn)$ with respect to the norm $\|\cdot\|_{m,G}$.
Define the Hilbert space for symmetric matrix-valued functions with square-integrable divergence:
\begin{equation}\label{defHdiv}
H(\Div,G;\bbS):=\{\tau\in L^2(G;\bbS)~:~\Div\tau\in L^2(G;\bbR^2)\},
\end{equation}
equipped with the following norm:
\begin{equation}\label{defHdivnorm}
\|\tau\|_{H(\Div,G)}:=\left(\|\tau\|_{0,G}^2+\|\Div\tau\|_{0,G}^2\right)^{\frac{1}{2}},\quad \text{for all}~ \tau\in H(\Div,G;\bbS).
\end{equation}

Suppose $\Omega$ is subdivided by a family of shape regular triangulations $\Th$ with~$h:=\max_{K\in\Th}h_K$ and $h_K:=\operatorname{diam}(K)$.
Let $\Eh$ be the set of all edges in~$\Th$, partitioned into boundary edges~$\EPh$ and interior edges~$\EIh$.
For any~$E\in\Eh$, define $h_E:=\operatorname{diam}(E)$ and fix a unit normal vector~$\nu_E$.
For~$\bbX\in\{\bbR,\bbRn,\bbS\}$, denote by $P_m(G;\bbX)$ the space of polynomials of degree~$\leq m$ over~$G$ with values in $\bbX$.
Let $H^1(\Th;\bbX)$ be the broken Sobolev space on $\Th$:
\begin{equation}\label{brokenH1}
    H^1(\Th;\bbX):=\bigg\{v\in L^2(\Omega;\bbX)~:~\sum_{K\in\Th}\big(\|v\|_{0,K}^2+|v|_{1,K}^2\big)<\infty\bigg\},
\end{equation}
equipped with the element-wise $H^1$-norm:
\begin{equation}\label{defbreakH1}
\|v\|_{1,\Th}:=\bigg(\sum_{K\in\Th}\left(\|v\|_{0,K}^2+|v|_{1,K}^2\right)\bigg)^{\frac{1}{2}},\quad \text{for all}~v\in H^1(\Th;\bbX).
\end{equation}

For two adjacent triangles $K^+, K^-\in\Th$ sharing an interior edge $E\in\EIh$, let~$\nu^+, \nu^-$ be the unit outwards normal vectors to the common edge $E$ of $K^+$ and~$K^-$, respectively.
For a vector-valued function $w$, write $w^+:=w|_{K^+}$ and~$w^-:=w|_{K^-}$. Then define the jump of $w$ as
\begin{equation}\label{defjump}
    [\![w]\!]:=\begin{cases}
    w^+(\nu^+\cdot\nu_E)+w^-(\nu^-\cdot\nu_E),&\text{ if } E\in\EIh,\\
    w,&\text{ if } E\in\EPh.
\end{cases}
\end{equation}
Throughout this paper, the notation $A\lesssim B$ denotes $A\leq CB$, where $C>0$ is a constant independent of the mesh size $h$. Similarly, $A\approx B$ signifies $A\lesssim B$ and~$B\lesssim A$.

\section{The Hu-Zhang element method}\label{tirhz}
In the construction of the Hu-Zhang element space, the characterization of bubble functions constitutes the pivotal step.
For each $K\in\Th$, define the local $H(\Div)$-conforming bubble function space:
\begin{equation}
\label{HdivbubbleK1}
\Sigma_{k, b}(K):=\{\tau \in P_k(K;\mathbb{S})~:~ \tau\cdot\nu|_{\partial K}=0\}.
\end{equation}
where $\nu$ is the unit outwards normal vector to $\partial K$.
Denote the vertices of $K$ by $\bmx_{K,0},\bmx_{K,1}$ and $\bmx_{K,2}$~(abbreviated as $\bmx_i$ when unambiguous). For each edge~$\bmx_i\bmx_j (i\neq j)$, let $t_{i,j}$ be the unit tangent vector along the edge. Define the symmetric matrices:
\begin{equation*}
    \bbT_{i,j}:=t_{i,j}t_{i,j}^T,\quad 0\leq i<j\leq 2.
\end{equation*}
As shown in \cite{hu2015finite,hu2014family}, the set $\{\bbT_{i,j}\}_{0\leq i<j\leq2}$ forms a basis for $\bbS$, and
\begin{equation}\label{HdivbubbleK1discret}
    \Sigma_{k, b}(K) = \sum_{0\leq i<j\leq 2}\lambda_i\lambda_j P_{k-2}(K;\bbR)\bbT_{i,j},
\end{equation}
where $\lambda_i$ is the associated barycentric coordinate corresponding to $\bmx_i$ for~$i=0,1,2$.
The global Hu-Zhang element space $\Sigh$ is defined as:
\begin{equation}
\label{H(div)T1}
\begin{aligned}
\Sigh:=\big\{\sigma \in &H(\Div, \Omega; \mathbb{S})~:~ \sigma=\sigma_{c}+\sigma_{b}, \sigma_{c} \in H^{1}(\Omega ; \mathbb{S}), \\
\sigma_{c}&|_{K} \in P_{k}(K ; \mathbb{S}),\sigma_{b}|_{K} \in \Sigma_{k, b}(K), ~\text{for all}~K \in \Th\big\},
\end{aligned}
\end{equation}
and the displacement space $\Vh$ as:
\begin{equation}
\label{2.4}
\Vh:=\left\{v \in L^{2}(\Omega; \mathbb{R}^{2})~:~v|_K \in P_{k-1}(K; \mathbb{R}^{2}),~\text{for all}~ K \in \Th\right\}.
\end{equation}
The Hu-Zhang element method for \eqref{eq1} seeks $(\sigmah,u_h)\in\Sigh\times\Vh$, such that
\begin{equation}\label{eq2}
\begin{cases}
a(\sigmah,\tauh) + b(\tauh,u_h)=0, &\text{for all }\tauh\in\Sigh,\\
b(\sigmah,\vh)=-(f,\vh),&\text{for all } \vh\in \Vh,
\end{cases}
\end{equation}
where $a(\cdot,\cdot)$ and $b(\cdot,\cdot)$ are defined in \eqref{abdefine}.
The error estimate of the Hu-Zhang element method is as follows \cite{hu2015finite,hu2014family}:
\begin{equation*}
\begin{aligned}
    \|\sigma-\sigmah\|_{0,\Omega}+h\|\sigma-\sigmah\|_{H(\Div,\Omega)}&+h\|u-\uh\|_{0,\Omega}\lesssim h^{k+1}(\|\sigma\|_{k+1,\Omega}+\|u\|_{k,\Omega}),
\end{aligned}
\end{equation*}
provided that $\sigma$ and $u$ are smooth enough.

Central to these results is the discrete inf-sup condition \cite{hu2015finite,hu2014family}, which guarantees: for any $\vh\in\Vh$, there exists a $\tauh\in\Sigh$ satisfying 
\begin{equation}\label{oldStablity}
\Div\tauh=\vh\,\text{ and }\,\|\tauh\|_{H(\Div,\Omega)}\lesssim\|\vh\|_{0,\Omega}.
\end{equation}
In this paper, a new inf-sup condition is established: there exists a $\tauh\in\Sigh$ satisfying~$\Div\tauh=\vh$ and $\|\tauh\|_{1,\Th}\lesssim\|\vh\|_{0,\Omega}$, 
which is the critical component in analyzing the discrete inf-sup condition for the curved Hu-Zhang element method.
The complete statement and rigorous proof are deferred to Section~\ref{diseretinfsupcondtion}.

\section{Triangulations with curved elements}\label{CTR} 
When the physical domain $\Omega$ has a curved boundary, the direct application of the Hu-Zhang element method on polygonal approximations of $\Omega$ (via triangular meshes) to solve the linear elasticity problem~\eqref{eq1} results in suboptimal convergence rates for both stress and displacement.
The loss in accuracy arises from the error of the geometric approximation inherent in piecewise linear representations of the boundary $\partial\Omega$.
To recover the optimal convergence rate, we employ higher-order geometric approximations of $\Omega$ through curved elements on the boundary.
 
We revisit the parametric framework for constructing a curvilinear triangulation~$\Thm$ ($m\geq1$), such that $\Omm:=\cup_{K\in\Thm}K$ is an approximate domain to~$\Omega$ with order $m$, as established in \cite{lenoir1986optimal} and described in \cite{arnold2020hellan}.
The methodology begins with a conforming, shape-regular, straight-edged triangulation~$\Tho$ with interior vertices belonging to $\Omega$ and boundary vertices belonging to~$\partial\Omega$.
Denote by $\TPho$ the set of boundary triangles containing at least one edge in~$\partial\Omo$.
We impose the following assumption.

\begin{Hy}
\label{H3.1}
Each triangle in $\Tho$ has at most two vertices on the boundary~$\partial\Omega$, and so it has at most one edge in $\partial\Omo$.
\end{Hy}

Next, for each $\Ko\in\Tho$, we define a polynomial diffeomorphism~$\FKm:\Ko\to\bbRn$ of degree $m$ that maps $\Ko$ onto a curvilinear triangle $\Km$.
This mapping is constructed as follows. All interior edges remain fixed, preserving conformity.
The boundary edge is determined by interpolating a parametric chart of $\partial\Omega$.
The interior points of $\Km$ are defined via the barycentric coordinate interpolation.
See \cite[eq.(14)]{lenoir1986optimal} for an explicit formula of~$\FKm\circ \widehat{\FKo}$, where $\widehat{F_\Ko}$ is the affine mapping from the reference triangle to $\Ko$.
The mapping~$\FKm$ satisfies the optimal~$L^\infty$-bound on its derivatives, as established in~\cite[Theorems 1-2]{lenoir1986optimal}. 
Furthermore, the curvilinear triangulation
\begin{equation*}
\Thm:=\{\Km=\FKm(\Ko)~:~\Ko\in\Tho\}
\end{equation*}
is a conforming, shape-regular triangulation of $\Omm$.
Extending the notations for~$m=1$ in Section \ref{tirhz}, denote by $\Ehm$ the set of all edges in triangulation~$\Thm$, partitioned into interior edges $\EIhm$ (all straight) and boundary edges~$\EPhm$~(possibly curved). 
Thus,~$\partial\Omm:=\bigcup_{\Em\in \EPhm}\Em$ is an approximate boundary to $\partial\Omega$ with order $m$.
The parametric mapping satisfies three conditions: 
1. $F_{K}^{1} \equiv \idI_{\Ko}$; 
2. if $\Ko$ has no edge on~$\partial\Omo$, then $\FKm \equiv \idI_{\Ko}$;
and 3. $\left.\FKm\right|_{E}=\left.\idI_{\Ko}\right|_{E}$ for all interior edges $E\in\EIhm$.
Denote by~$\TPhm$ the set of all boundary triangles and by~$\TIhm$ the set of all interior triangles.
Let $\OmmS:=\bigcup_{\Km\in\TPhm}\Km$ be the union of all boundary triangles.
It satisfies the inclusion 
\begin{equation*}
\OmmS\subset \Omega^m_h:=\{x\in\Omm~:~\exists\, y\in\partial\Omm, ~\operatorname{dist}(x,y)\leq h\},
\end{equation*}
where $\operatorname{dist}(x,y)$ is the distance between $x$ and $y$. Thus, by \cite[Lemma 2.2]{MR1313147}, the following inequality holds for $0\leq s \leq \frac{1}{2}$:
\begin{equation}\label{OmmSBys}
    \|v\|_{0,\OmmS}\lesssim h^s\|v\|_{s,\Omm},\quad \text{for all}~v\in H^s(\Omm).
\end{equation}

The local polynomial diffeomorphisms $\{\FKm\}_{\Ko\in\Tho}$ naturally induce a global piecewise polynomial mapping:
\begin{equation}\label{defFm}
\Fm:\Omo\to\Omm,\quad \Fm|_{\Ko}:=\FKm, \quad \text{for all}~\Ko\in\Tho.
\end{equation}
Furthermore, for $1\leq l\leq m<\infty$, we define the mapping between two approximate domains:~$\Philm:\Oml\to\Omm$, constructed piecewise by
\begin{equation}\label{PhilmDef}
\Philm|_{\Kl}:=\PhiKlm:=\FKm\circ(\FKl)^{-1}, \quad \text{for all}~\Kl\in\Thl.
\end{equation}

To compare the exact solution defined on the physical domain $\Omega$ and the numerical solution defined on the approximate domain $\Omm$, we establish an element-wise mapping from $\Omm$ to $\Omega$.
Specifically, for each element~$\Km\in\Thm$, we construct a diffeomorphic mapping~$\PsiKm:\Km\to\bbRn$ that transforms the parametric triangle $\Km$ into a curvilinear triangle $K$ exactly fitting $\Omega$.
Following the framework in \cite{lenoir1986optimal}, we construct this mapping as follows. For interior elements without boundary edges, the mapping reduces to the identity transformation.
For boundary elements containing an edge~$\Em\subset\partial\Omm$, the mapping is explicitly defined through the parametric equations in~\cite[eq.\,(32)]{lenoir1986optimal}, ensuring exact geometric matching.
This element-wise construction generates a curvilinear triangulation $\Th:=\{\PsiKm(\Km)\}_{\Km\in\Thm}$ exactly triangulating~$\Omega$. The local mappings $\PsiKm$ can be aggregated through a partition-of-unity approach to form a global mapping:
\begin{equation}\label{defPsim}
\Psim:\Omm\to\Omega,\quad \Psim|_{\Km}:=\PsiKm, \quad \text{for all}~ \Km\in\Thm.
\end{equation}

The physical domain $\Omega$ and its triangulation $\Th$ can be conceptually framed as the limiting case of $\{\Omm\}_{m\geq 1}$ and $\{\Thm\}_{m\geq1}$ as $m\to\infty$, respectively.
This motivates the adoption of alternative notations: $\Omega^{\infty}\equiv\Omega$, $\Th^\infty\equiv\Th$, $\Phi^{l\infty}\equiv\Psi^l$, $F_K^\infty\equiv\Psi^1$, etc., which will sometimes be convenient.

The following theorem, adapted from \cite[Theorem 3.2]{arnold2020hellan}, serves as the fundamental framework for both stability analysis and error estimates in the curved Hu-Zhang element method.
\begin{theorem}
\label{thcur}
Assume Hypothesis \ref{H3.1}, then for all $1\leq l \leq m\leq r$ and $m=\infty$, 
mappings $\Philm$ described above satisfy
\begin{equation}\label{curvePhi}
\begin{aligned}
 \|\nabla^s(\PhiKlm-\idI_{\Kl})\|_{L^\infty(\Kl)}&\lesssim h^{l+1-s},\quad \text{for all}~ 0\leq s \leq l+1,\\
 \|\nabla^s((\PhiKlm)^{-1}-\idI_{\Km})\|_{L^\infty(\Km)}&\lesssim h^{l+1-s},\quad \text{for all}~ 0\leq s \leq l+1, 
\end{aligned}
\end{equation}
where all constants depend on the piecewise $C^{r+1}$ norm of $\partial\Omega$.
\end{theorem}

We conclude this section with some basic results,
which follow from the chain rule and Theorem \ref{thcur}.
\begin{lemma}\label{propnormanddivError}
For all $1\leq l\leq m\leq r$ and $m=\infty$, consider functions $v\in H^s(\Km;\bbX)$, $\tau\in H^1(\Km;\bbS)$, $\hatv=v\circ\PhiKlm\in H^s(\Kl;\bbX)$ and~$\hattau=\tau\circ\PhiKlm\in H^1(\Kl;\mathbb{S})$, with $0\leq s\leq l+1$ and~$\bbX\in\{\bbR, \bbRn, \bbS\}$. Assume Hypothesis \ref{H3.1}, then
\begin{align}\label{chap::pre::L2norm}
    \|v\|_{s,\Km}\approx \|\hatv\|_{s,\Kl}&,\\
\label{chap::pre::Hdiveq1}
    \|(\Div\tau)\circ\PhiKlm-\hat{\Div}\hattau\|_{0,\Kl}\lesssim &~h^l|\hattau|_{1,\Kl},\\
\label{chap::pre::Hdiveq2}
    \|\Div\tau-(\hat{\Div}\hattau)\circ(\PhiKlm)^{-1}\|_{0,\Km}&\lesssim h^l|\tau|_{1,\Km},
\end{align}
where all constants depend on the piecewise $C^{r+1}$ norm of $\partial\Omega$.
\end{lemma}

\section{The curved Hu-Zhang element method}\label{curvehz}
This section generalizes the Hu-Zhang element to curved meshes through the mapping $\Fm$ in \eqref{defFm}, and establishes the well-posedness of the discrete formulation.
Hereafter, unless otherwise specified, we enforce Hypothesis~\ref{H3.1}. 

\subsection{The curved Hu-Zhang element space}\label{HZCurve} 
Let $m\geq1$ denote the order of the geometric approximation. 
For $k\geq3$, we introduce the $k$-th order finite element space on the curvilinear triangulation~$\Thm$, which begins with the finite element spaces~$\Sigho$ and $\Vho$ on $\Tho$, defined in~\eqref{H(div)T1} and~\eqref{2.4}, respectively.

The stress space $\Sighm$ over domain $\Omm$ is defined as
\begin{equation}
    \begin{aligned}
        \label{HdivKm}
    \Sighm:&=\{\sigma\in L^2(\Omm;\bbS)~:~\sigma=\sigma_c+\sigma_b,~\sigma_c\in H^1(\Omm;\bbS),\\
    &\sigma_c|_{\Km}\circ \FKm\in P_k(\Ko;\bbS),~\sigma_b|_{\Km}\in \Sigma_{k,b}(\Km),~\text{for all}~\Km\in\Thm\},
    \end{aligned}
\end{equation}
with $\Sigma_{k,b}(\Km):=\{\tau\in L^2(\Km;\bbS)\,:\, \tau\circ\FKm\in\Sigma_{k,b}(\Ko)\}$, where $\Sigma_{k,b}(\Ko)$ is the~$H(\Div,\Ko;\bbS)$ bubble function space in \eqref{HdivbubbleK1discret}.

When~$\Km$ is an interior element~($\FKm\equiv \idI_{\Ko}$),
$\Sigma_{k,b}(\Km)=\Sigma_{k,b}(\Ko)$ is the full~$H(\Div,\Km;\bbS)$ bubble function space. 
For the boundary element $\Km$, consider any interior edge $E\subset\partial\Km$, then the mapping $\FKm$ satisfies $\FKm|_E=\idI_{\Ko}|_E$. This induces $\tau\cdot\nu|_E=0$, for all $\tau\in\Sigma_{k,b}(\Km)$.
Consequently, the normal component of~$\tau$ vanishes on all interior edges.
This, combined with~\eqref{HdivKm}, indicates that $\Sighm$ is a conforming subspace of~$H(\Div,\Omm;\bbS)$.

It follows from \cite[Remark 3.1]{hu2015finite} that there is an interpolation operator $\Piho:H^1(\Omo;\bbS)\to\Sigho$, which satisfies
\begin{equation}\label{defPi}
    \bbo(\tau-\Piho\tau,\vh)=0,\quad \text{for all}~\tau\in H^1(\Omo;\bbS),\, \vh\in\Vho.
\end{equation}
For each $\Ko\in\Tho$, recall the mapping $\PhiKom:\Ko\to\Km\in\Thm$. Now, given~$\tau\in H^1(\Omm;\bbS)$, the global interpolation operator
$\Pih:H^1(\Omm;\bbS)\to\Sighm$ is defined element-wise by
\begin{equation}\label{PihmDef}
    \Pih\tau|_{\Km} = (\Piho\hattau|_{\Ko})\circ (\PhiKom)^{-1},
\end{equation}
with $\hattau:=\tau\circ\PhiKom$.
Using the techniques in \cite[Section 4.4]{brenner2008mathematical}, we have the following interpolation error estimate
\begin{equation}
    \label{interErrorPi}
    \|\tau-\Pih\tau\|_{0,\Omm}+h\|\tau-\Pih\tau\|_{1,\Thm}\lesssim h^{s}\|\tau\|_{s,\Omm},
\end{equation}
provided that $\tau\in H^s(\Omm;\bbS)$ with $1\leq s\leq k+1$.

The corresponding displacement space $\Vhm$ over $\Omm$ is defined as
\begin{equation}\label{VhKm}
    \Vhm:=\{v\in L^2(\Omm;\bbRn)~:~v\circ\Fm\in \Vho\}.
\end{equation}
Let $\Qh:L^2(\Omm;\bbR^2)\to\Vhm$ be the orthogonal projection operator with respect to the $L^2$-inner product:
\begin{equation}\label{Qhdefine}
    (\Qh u,\vh)_\Omm=(u,\vh)_\Omm,\quad\text{for all}~u\in L^2(\Omm;\bbRn),~ \vh\in\Vhm.
\end{equation}
By the orthogonality over each element $\Km$, it holds that
\begin{equation}\label{interErrorQh}
    \|u-\Qh u\|_{0,\Km}\lesssim h_{\Km}^s\|u\|_{H^s(\Km)},
\end{equation}
for all $u\in H^s(\Km;\bbRn),~0\leq s\leq k$.
The interpolation error estimates \eqref{interErrorPi} and~\eqref{interErrorQh} will be employed in Section~\ref{erroranalysis}.

\subsection{Mixed formulation of the curved Hu-Zhang element}
To construct the mixed variational formulation over the approximate domain~$\Omm$, we extend the bilinear forms in \eqref{abdefine} to parameter-dependent counterparts~$\aam(\cdot,\cdot)$ and $\bbm(\cdot,\cdot)$: 
\begin{equation*}
    \aam(\sigma,\tau):=(\maA\sigma,\tau)_{\Omm},\quad\bbm(\tau,v):=(\Div\tau,v)_{\Omm}.
\end{equation*}
The mixed finite element approximation to~\eqref{eq1} is: 
find~$(\sigmah,\uh)\in\Sighm\times\Vhm$ such that 
\begin{equation}
    \label{mixpuredisplacement}
    \begin{cases}
    \aam(\sigmah,\tauh) + \bbm(\tauh,\uh)=0,&\text{for all}~\tauh\in\Sighm,\\
    \bbm(\sigmah,\vh)=-(\tilde{f},\vh)_{\Omm},&\text{for all}~ \vh\in \Vhm,
    \end{cases}
\end{equation}
with $\tilde{f}=f\circ\Psim\det(\nabla\Psim)$ and $\Psim$ in \eqref{defPsim}.
The well-posedness of the discrete problem~\eqref{mixpuredisplacement}  will be rigorously established in the following section. Spaces and corresponding norms over the approximate domain $\Omm$, including~$H(\Div,\Omm;\bbS)$ and $H^1(\Thm;\bbX)$, are defined analogously to those in \eqref{defHdiv}--\eqref{defHdivnorm} and \eqref{brokenH1}--\eqref{defbreakH1}, respectively.
This subsection concludes with the following lemma, which is critical for both stability analysis and error estimates.
 
\begin{lemma}
    Let $1\leq l \leq r$ such that $m>l$, $\Phi\equiv\Philm:\Omega^l\to\Omm$ from \eqref{PhilmDef} with~$1<m\leq r$ and $m=\infty$. For any $\sigma,\tau\in H(\Div,\Omm;\bbS)$ and $v\in L^2(\Omm;\bbRn)$, define their pullbacks $\hatsigma=\sigma\circ\Phi$, $\hattau=\tau\circ\Phi$ and $\hatv=v\circ\Phi$. The following estimate holds:
    \begin{equation}\label{amldiff}
        |\aam(\sigma,\tau)-a^l(\hatsigma,\hat{\tau})|\lesssim h^l\|\hatsigma\|_{0,\OmlS}\|\hat{\tau}\|_{0,\OmlS}.
    \end{equation}
    Furthermore, if $\sigma\in H(\Div,\Omm;\bbS)\cap H^1(\Thm;\bbS)$,
    \begin{equation}\label{bmldiff}
        |\bbm(\sigma,v)-b^l(\hatsigma,\hatv)|\lesssim h^l\|\hat{\nabla}_h\hatsigma\|_{0,\OmlS}\|\hatv\|_{0,\OmlS},
    \end{equation}
    and if $v\in H^1(\Thm;\bbRn)$,
    \begin{equation}\label{bmldiffMo}
        \begin{aligned}
            |\bbm(\sigma,v)-&b^l(\hatsigma,\hatv)|\lesssim h^l\|\hatsigma\|_{0,\OmlS}\|\hat{\nabla}_h\hatv\|_{0,\OmlS}\\
            & +h^l\bigg(\sum_{\El\in\EPhl}h_{\El}\|\hatsigma\hat{n}\|_{L^2(\El)}^2\bigg)^{\frac{1}{2}}\bigg(\sum_{\El\in\EPhl}h^{-1}_{\El}\|\hatv\|_{L^2(\El)}^2\bigg)^{\frac{1}{2}},
        \end{aligned}
    \end{equation}
where $\HatNab_h$ is the broken gradient operator on $\Thl$.
\end{lemma}
\begin{proof}
Let $J:=\HatNab_h\Phi$ be the element-wise Jacobian.
By transforming the integrals of the bilinear forms $\aam(\sigma,\tau)$ and $\bbm(\sigma,v)$ from~$\Omm$ to $\Omega^l$, we derive the following equations on each element:
\begin{equation}\label{amldiffCs}
    (\maA\sigma,\tau)_{\Km}-(\maA\hatsigma,\hattau)_{\Kl}=(\maA\hatsigma,\hattau(\det(J)-1))_{\Kl},
\end{equation}
and
\begin{equation}\label{bmldiffCs}
\begin{aligned}
    (\Div\sigma,v)_{\Km}&-(\HatDiv\hatsigma,\hatv)_{\Kl}\\&=(\widehat{\Div\sigma}-\HatDiv\hatsigma,\hatv\det(J))_{\Kl}+(\HatDiv\hatsigma,\hatv(\det(J)-1))_{\Kl}.
\end{aligned}
\end{equation}
Applying integration by parts and using $\Phi|_E=\idI_{\Kl}|_E$ on interior edges, we obtain:
\begin{equation}\label{bmlMdiffCs}
\begin{aligned}
    (\Div\sigma&,v)_{\Km}-(\HatDiv\hatsigma,\hatv)_{\Kl}\\
    &=-\big(\hatsigma,[\widehat{\varepsilon(v)}-\hat{\varepsilon}(\hatv)]\det(J)\big)_{\Kl}-\big(\hatsigma,\hat{\varepsilon}(\hatv)(\det(J)-1)\big)_{\Kl}\\
    &+\big(\hatsigma(J^{-T}-\bdI)\hat{n},\hatv\det(J)\big)_{\partial\Kl\cap\partial\Oml}+\big(\hatsigma\hat{n},\hatv(\det(J)-1)\big)_{\partial\Kl\cap\partial\Oml}.
\end{aligned}
\end{equation}
Since $J=\bdI$ holds for all interior triangles, the estimates in \eqref{amldiff}, \eqref{bmldiff} and \eqref{bmldiffMo} follow from combining~\eqref{amldiffCs}, \eqref{bmldiffCs} and \eqref{bmlMdiffCs} with the parametric bound in \eqref{curvePhi} and applications of the weighted Cauchy-Schwarz inequality.
\end{proof}

\subsection{Well-posedness}
Following the abstract stability theory for mixed finite element methods \cite{boffi2013mixed}, we need to verify two fundamental conditions:
\begin{enumerate}
\item K-ellipticity.
\begin{equation}
\label{chap::well::K}
\aam(\tau,\tau)\gtrsim\|{\tau}\|^2_{H(\Div,\Omm)},\quad\text{for all}~\tau\in\Zhm,
\end{equation}
where the discrete kernel space $\Zhm$ is defined as:
\begin{equation}
\label{Zh}
\Zhm:=\{\tau\in \Sighm~:~ \bbm(\tau, v)=0,~\text{for all}~v\in \Vhm\}.  
\end{equation}
\item Discrete inf-sup condition.
\begin{equation}
\label{chap::well::B}
\sup_{0\neq\tau\in \Sighm}\frac{\bbm(\tau,v)}{\|{\tau}\|_{H(\Div,\Omm)}}\gtrsim \|{v}\|_{0,\Omm},\quad \text{for all}~ v\in \Vhm. 
\end{equation}
\end{enumerate}

\subsubsection{K-ellipticity}
 The following lemma establishes the K-ellipticity of $\aam(\cdot,\cdot)$ over the discrete kernel space $\Zhm$.
\begin{lemma}\label{IPhdinvBoundLemma}
    For any $\tau\in\Sighm$, it holds that
    \begin{equation}\label{IPhdivBound}
        \|(\idI-\Qh)\Div\tau\|_{0,\Km}\lesssim \min\big\{h|\tau|_{1,\Km},\,\|\tau\|_{0,\Km}\big\},\quad \text{for all}~ \Km\in\Thm.
    \end{equation}
\end{lemma}
\begin{proof}
    For any interior element $\Km\in\TIhm$, we have $\Km=\Ko$ and $\tau|_{\Km}\in P_k(\Km;\bbS)$. This implies~$\Div\tau|_{\Km}\in P_{k-1}(\Km;\bbRn)$.
    Since $\Qh$ preserves polynomials of degree~$k-1$, one obtains $(\idI-\Qh)\Div\tau|_{\Km}=0$, so that \eqref{IPhdivBound} is satisfied trivially.
    For any boundary element $\Km\in\TPhm$, define $\hattau:=\tau|_{\Km}\circ\FKm$ and $v:=\HatDiv\hattau\circ(\FKm)^{-1}$. 
    Then
    \begin{equation*}
    \begin{aligned}
        \|(\idI-\Qh)\Div\tau\|_{0,\Km}&=\|(\idI-\Qh)(\Div\tau-v)\|_{0,\Km}\\
        &\leq \|\Div\tau-\HatDiv\hattau\circ(\FKm)^{-1}\|_{0,\Km}.
    \end{aligned}
    \end{equation*}
    The combination of the estimate in \eqref{chap::pre::Hdiveq2} with $l=1$ and the inverse estimate \cite[Lemma 4.5.3]{brenner2008mathematical} establishes \eqref{IPhdivBound} for boundary elements, which completes the proof.
\end{proof}
The definition of the kernel space \eqref{Zh} implies that 
$
(\idI-\Qh)\Div\tau=\Div\tau
$ holds for all $\tau\in\Zhm$.
This, combined with Lemma \ref{IPhdinvBoundLemma}, yields
$
\|\Div\tau\|_{0,\Omm}\lesssim \|\tau\|_{0,\Omm}
$. Consequently, the K-ellipticity condition \eqref{chap::well::K} is verified.
 
\subsubsection{Discrete inf-sup condition}\label{diseretinfsupcondtion}
For $m\geq2$, the non-polynomial structure of $\Sighm$ and $\Vhm$ over the boundary element~$\Km$ introduces intrinsic difficulties in analyzing the discrete inf-sup condition.
This challenge originates from the inherent difficulty in establishing a precise characterization of the divergence space associated with the bubble function space $\Sigma_{k,b}(\Km)$, analogous to the result in \cite{hu2015finite,hu2014family}.

The key to overcoming the aforementioned challenge lies in the following theorem.
By lifting the $H(\Div)$-norm in \eqref{oldStablity} to the broken $H^1$-norm, this theorem enables us to establish the discrete inf-sup condition on $\Thm$ through the error estimate between~$\bbm(\cdot,\cdot)$ and $\bbo(\cdot,\cdot)$.
The proof follows a methodology similar to that in \cite{chen2018fast,hu2015finite}, and can be divided into two main steps: see Lemma~\ref{trihz::lembubblespacepro} and Lemma~\ref{trihz::lemconT} below for details.
It is worth noting that the inf-sup condition in mesh-dependent norms on simplicial meshes introduced in~\cite{chen2018fast}, which is extended here to curved meshes, will be employed in deriving the~$L^2$-error estimate for the stress in Section \ref{L2errorstress}.

\begin{theorem}\label{trihz::importBB}
For any $\vh\in \Vho$, there exists a $\tauh\in\Sigho$, such that
\begin{equation}\label{InfsupImportantance}
    \Div\tauh=\vh\quad\text{and}\quad \|\tauh\|_{1, \Tho}\lesssim \|\vh\|_{0,\Omo}.
\end{equation}
\end{theorem}

Before proving Theorem \ref{trihz::importBB}, we introduce the following rigid motion space on each element~$\Ko\in\Tho$.
\begin{equation*}
R(\Ko):=\{v\in H^1(\Ko;\bbRn)~:~ (\nabla v+\nabla v^T)/2=0\}.
\end{equation*}
It follows from the definition that $R(\Ko)$ is a subspace of $P_1(\Ko;\bbRn)$. This allows for defining the orthogonal complement space of $R(\Ko)$ with respect to~$P_{k-1}(\Ko;\bbRn)$ by 
\begin{equation*}
R^\perp(\Ko):=\{v\in P_{k-1}(\Ko;\bbRn)~:~(v,w)_{\Ko}=0, ~\text{for all}~ w\in R(\Ko)\}.
\end{equation*}
According to \cite[Theorem 2.2]{hu2015finite}, the divergence space of~$\Sigma_{k,b}(\Ko)$ satisfies
\begin{equation*}
   \Div\Sigma_{k,b}(\Ko)=R^\perp(\Ko), \quad \text{for all}~\Ko\in\Tho.
\end{equation*}
This intrinsic property facilitates the subsequent lemma concerning discrete stability.
\begin{lemma}\label{trihz::lembubblespacepro}
    For any $v\in R^\perp(\Ko)$, there exists a $\tau\in\Sigma_{k,b}(\Ko)$, such that
    \begin{equation*}
        \Div\tau=v\quad \text{and}\quad \|\tau\|_{1,\Ko}\lesssim \|v\|_{0,\Ko}.
    \end{equation*}
\end{lemma}
\begin{proof}
Choosing $\tau\in\Sigma_{k,b}(\Ko)$, such that
\begin{equation*}
 \Div\tau=v\,\text{ and }\, \|\tau\|_{0,\Ko}=\min\{\|\tau\|_{0,\Ko}~:~\Div\tau=v,~\tau\in\Sigma_{k,b}(\Ko)\}.
\end{equation*}
Let $F$ be the affine mapping from reference element $\hat{K}$ to $\Ko$, denoted by $F(\hat{x})=B\hat{x}+b$, where~$B$ is an invertible matrix and $b$ is a translation vector determined by $\Ko$.
Let~$\hat{\tau}\in\Sigma_{k,b}(\hat{K})$ be~$\hat{\tau}(\hat{x}):=B^{-1}\tau(x)B^{-T}$ with $x=F(\hat{x})$.
A direct computation shows that~$
\Div\tau(x)=B\HatDiv\hat{\tau}(\hat{x})$.
The invertibility of $B$ and the definition of $\tau$ imply that~$\|\HatDiv\hattau\|_{0,\hat{K}}$ is a well-defined norm. Applying the norm equivalence theorem (valid for finite-dimensional spaces) and a scaling argument, we derive~$\|\tau\|_{1,\Ko}\lesssim\|\Div\tau\|_{0,\Ko}=\|v\|_{0,\Ko}$, which completes the proof.
\end{proof}

By \cite[Lemma 3.1]{hu2015finite}, for any $\vh\in \Vho$, there exists a $\tauh\in\Sigho\cap H^1(\Omo;\bbS)$ satisfying the orthogonality condition:
\begin{equation}\label{trihz::orthogonalityPrp}
    (\Div\tauh-\vh,p)_{\Ko}=0,\quad\text{for all}~p\in R(\Ko),~\Ko\in\Tho,
\end{equation}
with $\|\tauh\|_{H(\Div,\Omo)}\lesssim \|\vh\|_{0,\Omo}$.
One observation is that the stability persists when replacing the $H(\Div)$-norm with the stronger~$H^1$-norm, i.e., $\|\tauh\|_{1,\Omo}\lesssim\|\vh\|_{0,\Omo}$.
This stems from the proof in~\cite[Lemma 3.1]{hu2015finite} where both parts of $\tauh$ (interpolation part and correction part) can be controlled by $\|\vh\|_{0,\Omo}$.
These results establish the following lemma to control the rigid motion part of $\vh$.
\begin{lemma}\label{trihz::lemconT}
    For any $\vh\in \Vho$, there exists a $\tauh\in\Sigho\cap H^1(\Omo;\bbS)$ such that
    \begin{equation*}
        (\Div\tauh-\vh, p)_{\Ko} =0\quad\text{ and }\quad \|\tauh\|_{1,\Omo}\lesssim \|\vh\|_{0,\Omo},
    \end{equation*}
    for all~$p\in R(\Ko)$ and $\Ko\in\Tho$.
\end{lemma}
Now we prove Theorem \ref{trihz::importBB} in two steps.
First, select $\tau_1\in\Sigho$ via Lemma \ref{trihz::lemconT}, such that 
\begin{equation*}
(\vh-\Div\tau_1)|_{\Ko}\in R^\perp(\Ko),\quad\text{for all}~\Ko\in\Tho.
\end{equation*}
Second, by Lemma \ref{trihz::lembubblespacepro}, there exists a $\tau_2\in\Sigho$ with $\tau_2|_{\Ko}\in\Sigma_{k,b}(\Ko)$ satisfying
\begin{equation*}
\Div\tau_2=\vh-\Div\tau_1\quad\text{ and }\quad \|\tau_2\|_{1,\Tho}\lesssim \|\vh-\Div\tau_1\|_{0,\Omo}.
\end{equation*}
The sum $\tauh=\tau_1+\tau_2$ then verifies
\begin{equation*}
\Div\tauh=\vh\quad\text{ and }\quad \|\tauh\|_{1,\Tho}\lesssim \|v\|_{0,\Omo},
\end{equation*}
which completes the proof of Theorem \ref{trihz::importBB}.

We are in the position to show the discrete inf-sup condition \eqref{chap::well::B} for $m>1$.
\begin{theorem}\label{chap::DisInfSup}
For any $v\in \Vhm$, there exists a $\tau\in\Sighm$ such that for all sufficiently small mesh size $h$,
\begin{equation}\label{curveInfSup}
\bbm(\tau,v)\gtrsim\|{v}\|_{0,\Omm}^2\quad \text{and}\quad\|{\tau}\|_{1,\Thm}\lesssim\|{v}\|_{0,\Omm},
\end{equation}
where all constants are independent of both $h$ and $m$. 
Consequently, the inf-sup condition~\eqref{chap::well::B} holds.
\end{theorem}
\begin{proof}
For any $v\in \Vhm$, take $\hatv=v\circ \Fm\in \Vho$. By Theorem \ref{trihz::importBB}, there is a $\hattau\in\Sigho$ satisfying
\begin{equation}
    \label{T1B}
    \HatDiv\hattau=\hatv\quad \text{and}\quad \|{\hattau}\|_{1,\Tho}\leq C\|{\hatv}\|_{0,\Omo}.
\end{equation}
Choose $\tau=\hattau\circ(\Fm)^{-1}$, then it follows from \eqref{bmldiff} and \eqref{chap::pre::L2norm} with $l=1$ that
\begin{equation*}
\begin{aligned}
\bbm(\tau,v)
&\geq \|\hatv\|_{0,\Omo}^2-Ch\|\hattau\|_{1,\Tho}\|\hatv\|_{0,\Omo}\geq (1-Ch)\|v\|_{0,\Omm}^2.
\end{aligned}
\end{equation*}
For sufficiently small $h$, we obtain the first inequality in~\eqref{curveInfSup}. The second inequality in~\eqref{curveInfSup} follows from
$
\|\tau\|_{1,\Thm}\lesssim\|\hattau\|_{1,\Tho}\lesssim\|v\|_{0,\Omm},
$
which completes the proof.
\end{proof}

\begin{remark}
The discrete mixed formulation \eqref{mixpuredisplacement} can be reformulated in terms of the monolithic bilinear form:
\begin{equation}\label{mixfemA}
    \bbAm(\sigmah,\uh;\tauh,\vh)=-(\tilde{f},\vh)_{\Omm},
\end{equation}
where $\bbAm(\sigmah,\uh;\tauh,\vh) := \aam(\sigmah,\tauh) + \bbm(\tauh,\uh) + \bbm(\sigmah,\vh)$.
The K-ellipticity~\eqref{chap::well::K} and the discrete inf-sup condition \eqref{chap::well::B} yield the stability,
\begin{equation}\label{prioriEst}   \|\sigmah\|_{H(\Div,\Omm)}+\|\uh\|_{0,\Omm}\lesssim\sup_{(\tauh,\vh)\in\Sighm\times\Vhm}\frac{\bbAm(\sigmah,\uh;\tauh,\vh)}{\|\tauh\|_{H(\Div,\Omm)}+\|\vh\|_{0,\Omm}}.
\end{equation} 
This inequality will be used in the error estimate derived in the next section.
\end{remark}

\section{Error analysis}\label{erroranalysis}
The error analysis of the curved Hu-Zhang element method faces an intrinsic difficulty: the divergence space of the curved Hu-Zhang element is not contained in the discrete displacement space (i.e.,~$\Div\Sighm\not\subseteq\Vhm$),
which precludes a direct error estimate of stress in the $L^2$-norm through standard arguments like those for simplicial meshes in \cite[Remark 3.1]{hu2015finite}.
To circumvent this, we introduce mesh-dependent norms in the $L^2$-error analysis. 
It is noteworthy that such non-inclusion property introduces an error which reduces the convergence rate of stress by half an order.
Crucially, the suboptimality, rooted in the incompatibility of the approximation space, can be improved by increasing degree of the polynomial on boundary elements in the discretization scheme.
This adjustment restores the optimal convergence rate, as numerically demonstrated in the next section.

Compared with the case that $\Omega$ is a polygonal domain, the geometric approximation of the curved boundary $\partial\Omega$ introduces a consistency error, which in turn deteriorates the final convergence rate of the finite element solution.
To analyze such a consistency error,
we will use the following inequality, which corresponds to $s=\frac{1}{2}$ in~\eqref{OmmSBys}:
\begin{equation}\label{OmmSBysused}
    \|v\|_{0,\OmmS}\lesssim h^{\frac{1}{2}}\|v\|_{\frac{1}{2},\Omm},\quad \text{for all}~v\in H^{\frac{1}{2}}(\Omm).
\end{equation}

\subsection{The error estimate by the geometric approximation}
We begin with estimating the error caused by the geometric approximation, through substituting the exact solution into the variational formulation defined on the approximate domain $\Omm$.
\begin{lemma}\label{geoError}
    Recall the mapping $\Psi^m:\Omm\to\Omega$ with $\Psi^m_K:=\Psi^m|_K$ from Section~\ref{CTR}. 
    Let~$(\sigma,u)$ be the solution of \eqref{eq1} on the physical domain $\Omega$.
    Denote by $\hatsigma:=\sigma\circ\Psi^m$ and $\hatu:=u\circ\Psi^m$ the pull-back of $\sigma$ and $u$, respectively.
    Suppose~$u\in H^{s+1}(\Omega;\bbRn)$ with $s\geq\frac{3}{2}$, such that $\sigma\in H^{s}(\Omega;\bbS)$.
    Then we obtain the equation for the error of the geometric approximation:
    \begin{equation}\label{geoerr1}
    \bbAm(\hatsigma,\hatu;\hattau,\hatv)=-(\tilde{f},\hatv)_{\Omm}+E(\hattau,\hatv),
    \end{equation}
    for all $\hattau=\tau\circ\Psi^m$ and $\hatv=v\circ\Psi^m$ with $(\tau,v)\in\Sigma\times V$, 
    where
    \begin{equation}\label{geoerr2}
        |E(\hat{\tau},\hatv)|\lesssim h^{m+\frac{1}{2}}(\|\hattau\|_{0,\OmmS}+\|\hatv\|_{0,\OmmS})\|u\|_{H^{\frac{5}{2}}(\Omega)}.
    \end{equation}
\end{lemma}
\begin{proof}
    The variational formulation \eqref{eq1} and the definition of $\bbAm$ lead to
    \begin{equation*}
    \begin{aligned}
        E(\hat{\tau},\hatv)&=\left[\aam(\hatsigma,\hattau)-a(\sigma,\tau)\right]+ \left[\bbm(\hattau,\hatu) - b(\tau,u)\right]+\left[\bbm(\hatsigma,\hatv) - b(\sigma,v)\right].
    \end{aligned}
    \end{equation*}
    Using the estimates from \eqref{amldiff}, \eqref{bmldiff} and \eqref{bmldiffMo} through the parameter substitution~$(m,l)\to(\infty, m)$, and applying the boundary condition $u|_{\partial\Omega}=0$, we derive the consistency error:
    \begin{equation*}
     |E(\hat{\tau},\hatv)|\lesssim h^m(\|\hattau\|_{0,\OmmS}+\|\hatv\|_{0,\OmmS})(\|\hatsigma\|_{H^1(\OmmS)}+\|\hat{u}\|_{H^1(\OmmS)}).
    \end{equation*}
    This, combined with \eqref{chap::pre::L2norm} and \eqref{OmmSBysused}, completes the proof.
\end{proof}

For technical reasons, we impose a relatively strong regularity assumption.
In fact, similar results can still be obtained when the solution has lower regularity, with a possible loss of half an order. 
Therefore, in the following analysis, we always assume that $u\in H^{s+1}(\Omega;\bbRn)$ with $s \geq \frac{3}{2}$.

\subsection{The error estimate of stress in the \texorpdfstring{$H(\Div)$}{\textit{H(div)}}-norm}
Lemma \ref{geoError} establishes a direct method for estimating the errors of stress in the~$H(\Div)$-norm and displacement in the $L^2$-norm.
\begin{theorem}\label{standErrorThe}
Let $(\sigma, u)$ be the solution of \eqref{eq1},
and $(\hatsigmah,\hatuh)\in\Sighm\times\Vhm$ be the solution of \eqref{mixfemA}.
Then, the following error estimate holds:
    \begin{equation}\label{standError}
    \begin{aligned}
    \|\hatsigma-\hatsigmah\|_{H(\Div,\Omm)}+\|\hatu-\hatuh\|_{0,\Omm}&\lesssim h^q\|u\|_{H^{q+2}(\Omega)},
    \end{aligned}
    \end{equation}
    with the convergence rate $q = \min\{k, m+\frac{1}{2}, s-1\}$.
\end{theorem}
\begin{proof}
    Starting from the variational formulation in \eqref{mixfemA} and \eqref{geoerr1}, for any $(\hattauh,\hatvh)\in\Sighm\times\Vhm$, we derive the following equation:
    \begin{equation}\label{standErrorImportEq}
         \bbAm(\hatsigmah,\hatuh;\hattauh,\hatvh) = \bbAm(\hatsigma,\hatu;\hattauh,\hatvh) - E(\hattauh,\hatvh).
    \end{equation}
    Take $\hatsigma_I=\Pih\hatsigma$ and $\hatu_I=\Qh\hatu$, where $\Pih$ and $\Qh$ are defined in \eqref{PihmDef} and \eqref{Qhdefine}, respectively.
    Subtracting $\bbAm(\hatsigma_I,\hatu_I;\hattauh,\hatvh)$ from \eqref{standErrorImportEq}, we can get:
    \begin{equation*}
    \bbAm(\hatsigmah-\hatsigma_I,\hatuh-\hatu_I;\hattauh,\hatvh) = \bbAm(\hatsigma-\hatsigma_I,\hatu-\hatu_I;\hattauh,\hatvh) - E(\hattauh,\hatvh).
    \end{equation*}
    This, combined with the stability \eqref{prioriEst} and the consistency error~\eqref{geoerr2}, yields
    \begin{equation*}
    \begin{aligned}
         \|\hatsigmah-\hatsigma_I\|&_{H(\Div,\Omm)}+\|\hatuh-\hatu_I\|_{0,\Omm}\\
         &\lesssim \|\hatsigma-\hatsigma_I\|_{H(\Div,\Omm)} +\|\hatu-\hatu_I\|_{0,\Omm} + h^{m+\frac{1}{2}}\|u\|_{H^{\frac{5}{2}}(\Omega)}.
    \end{aligned}
    \end{equation*}
    The final error estimate \eqref{standError} follows from the triangle inequality and the interpolation error estimates in \eqref{interErrorPi} and \eqref{interErrorQh}.
\end{proof}

\subsection{The error estimates in mesh-dependent norms}\label{L2errorstress}
The lack of the inclusion $\Div\Sighm\not\subseteq\Vhm$ makes it difficult to analyze the error of stress in the $L^2$-norm.
The key ingredient here is to prove that the monolithic bilinear form $\bbAm(\cdot,\cdot;\cdot,\cdot)$ is stable on $\Sighm\times\Vhm$ in the mesh-dependent norms~\eqref{def0hmand1hm} below, where the case for~$m=1$ was proved in \cite[Lemma 3.3]{chen2018fast}.

\subsubsection{The stability in mesh-dependent norms}
For any $\tauh\in\Sighm$ and $\vh\in\Vhm$, define the mesh-dependent norms:
\begin{equation}\label{def0hmand1hm}
    \begin{aligned} \|\tauh\|_{0,h,m}^2:&=\|\tauh\|_{0,\Omm}^2+\sum_{E\in\Fhm}h_E\|\tauh\nu_E\|_{0,E}^2,\\
    |\vh|_{1,h,m}^2:&=\|\varepsilonh(\vh)\|_{0,\Omm}^2+\sum_{E\in\Fhm}h_E^{-1}\|[\![\vh]\!]\|_{0,E}^2.
    \end{aligned}
\end{equation}
Here, $\varepsilonh(\cdot)$ is the element-wise symmetric gradient operator and $[\![\cdot]\!]$ is defined in \eqref{defjump}. 
For $k\geq3$, it follows from \cite[Lemma 3.3]{chen2018fast} that:
 \begin{equation}\label{meshnorms1}
    |\vh|_{1,h,1}\lesssim\sup_{0\neq\tauh\in\Sigho}\frac{\bbo(\tauh,\vh)}{\|\tauh\|_{0,h,1}},\quad \text{for all}~\vh\in\Vho.
\end{equation}
This paper generalizes \eqref{meshnorms1} to curved meshes.
Before proving it, we need the following result which can be obtained by combining \cite[eq.(3.6), eq.(3.8)]{chen2018fast}.
\begin{equation}\label{H1NormBound}
    \|\nabla_h v\|_{0,\Omo}^2+\|v\|_{0,\Omo}^2\lesssim |v|_{1,h,1}^2,\quad \text{for all}~ v\in H^1(\Tho;\bbRn).
\end{equation}

\begin{lemma}
    For $k\geq3$ and $h$ small enough, the following inf-sup condition holds:
    \begin{equation}\label{bmhimport}
    |\vh|_{1,h,m}\lesssim\sup_{0\neq\tauh\in\Sighm}\frac{\bbm(\tauh,\vh)}{\|\tauh\|_{0,h,m}},\quad \text{for all}~ \vh\in\Vhm.
\end{equation}
\end{lemma}
\begin{proof}
We require the following three estimates:
\begin{align}
    \label{normEquiv}
    \|\tauh\|_{0,h,m}&\approx\|\hattauh\|_{0,h,1},\quad |\vh|_{1,h,m}\approx|\hatvh|_{1,h,1},\\
    \label{bmboundedby0hmand1hm}
    \bbm(\tauh,\vh)&\geq b^1(\hattauh,\hatvh)-Ch\|\tauh\|_{0,h,m}|\vh|_{1,h,m},
\end{align}
with $\hattauh=\tauh\circ\Philm$, $\hatvh=\vh\circ\Philm$ for $l=1$. Let $J=\HatNab_h\Philm$ be defined element-wise. 
On each element, the following equations hold:
\begin{equation}\label{tauhml}
\|\tauh\|_{0,\Km}^2-\|\hattauh\|_{0,\Kl}^2=(|\hattauh|^2,\det(J)-1)_{\Kl},
\end{equation}
and
\begin{equation}\label{tauhnuml}
\begin{aligned}
    &h_{\Em}\|\tauh\nu_{\Em}\|_{0,\Em}^2-h_{\El}\|\hattauh\nu_{\El}\|_{0,\El}^2\\
    =&(h_{\Em}-h_{\El})\|\hattauh\nu_{\El}\|_{0,\El}^2+h_{\Em}\big(|\hattauh\nu_{\El}|^2,\ |J^{-T}\nu_{\El}|^{-1}\det(J)-1\big)_{\El}\\
    &+h_{\Em}\big(|\hattauh J^{-T}\nu_{\El}|^2-|\hattauh\nu_{\El}|^2,\ |J^{-T}\nu_{\El}|^{-1}\det(J)\big)_{\El}.
\end{aligned}    
\end{equation}
Utilizing the estimate $\|J-\bdI\|_{L^\infty(\Kl)}\lesssim h^l$ from \eqref{curvePhi}, we bound the terms on the right hand side of \eqref{tauhml} and \eqref{tauhnuml}:
\begin{equation*}
\left|\, \|\tauh\|_{0,h,m}^2-\|\hattauh\|_{0,h,l}^2\,\right|\lesssim h^l\|\hattauh\|_{0,h,l}^2.
\end{equation*}
This establishes the first equation in \eqref{normEquiv}.
A similar argument with the discrete Sobolev inequality \eqref{H1NormBound} derives the second equation in \eqref{normEquiv}.
The inequality~\eqref{bmboundedby0hmand1hm} follows from~\eqref{bmldiffMo} and~\eqref{H1NormBound}.
The inf-sup condition \eqref{bmhimport} is obtained by the combination of \eqref{normEquiv},\eqref{bmboundedby0hmand1hm} and~\eqref{meshnorms1}.
\end{proof}

\begin{remark}
From \eqref{chap::pre::L2norm} and the second equation in \eqref{normEquiv}, it follows that \eqref{H1NormBound} also holds over $\Omm$:
    \begin{equation}\label{H1NormBoundm}
    \|\nabla_h v\|_{0,\Omm}^2+\|v\|_{0,\Omm}^2\lesssim |v|_{1,h,m}^2,\quad \text{for all}~ v\in H^1(\Thm;\bbRn).
\end{equation}
\end{remark}

By the uniform boundedness of $\maA$, the trace theorem \cite[Theorem 1.6.6]{brenner2008mathematical} and the inverse estimate, the bilinear form $\aam(\cdot,\cdot)$ is coercive over the full discrete space:
\begin{equation*}
    \aam(\sigmah,\sigmah)\gtrsim \|\sigmah\|_{0,h,m}^2, \quad  \text{for all}~ \sigmah\in\Sighm.
\end{equation*}
Combining this coercivity with the discrete inf-sup condition \eqref{bmhimport}, we obtain the following stability.
\begin{theorem}
    For $k\geq 3$ and $h$ small enough, it holds that 
    \begin{equation}\label{Astabnorm2}
        \|\sigmah\|_{0,h,m}+|\uh|_{1,h,m}\lesssim \sup_{(\tauh,\vh)\in\Sighm\times\Vhm}\frac{\bbAm(\sigmah,\uh;\tauh,\vh)}{\|\tauh\|_{0,h,m}+|\vh|_{1,h,m}},
    \end{equation}
    for any~$(\sigmah,\uh)\in\Sighm\times\Vhm$.
\end{theorem}

\subsubsection{The error estimate of stress in the \texorpdfstring{$L^2$}{\textit{L\^{}2}}-norm}
The following lemma estimates the error arising from the non-inclusion, i.e., $\Div\Sighm\not\subseteq\Vhm$.
This results in a half-order reduction in the convergence rate of stress compared to the optimal one.
\begin{lemma}
For any $\tau\in\Sighm$, the following estimate holds,
    \begin{equation}\label{bmIpErrorFirst}
        |\bbm(\tau, (\idI-\Qh)u)|\lesssim \min\big\{h|\tau|_{1,\TPhm},\,\|\tau\|_{0,\OmmS}\big\}\|(\idI-\Qh)u\|_{0,\OmmS}.
    \end{equation}
    Moreover, if $u\in H^{s+1}(\Omm;\bbRn)$,
    \begin{equation}\label{bmIpError}
        |\bbm(\tau, (\idI-\Qh)u)|\lesssim  \min\big\{h|\tau|_{1,\TPhm},\,\|\tau\|_{0,\OmmS}\big\}h^q\|u\|_{H^{q}(\Omm)},
    \end{equation}
    with $q=\min\{k+\frac{1}{2},s+1\}$ and $|\tau|_{1,\TPhm}=\big(\sum_{\Km\in\TPhm}|\tau|_{1,\Km}^2\big)^{\frac{1}{2}}$.
\end{lemma}
\begin{proof}
    By the projection property of $\Qh$ and $(\idI-\Qh)\Div\tau=0$ on all interior elements, we derive
    \begin{equation*}
    \begin{aligned}
        \bbm(\tau,(\idI-\Qh)u)=\sum_{\Km\in\TPhm}((\idI-\Qh)\Div\tau, (\idI-\Qh)u)_{\Km}.
    \end{aligned}  
    \end{equation*}  
    Applying the Cauchy-Schwarz inequality and the estimate in Lemma~\ref{IPhdinvBoundLemma}, we establish~\eqref{bmIpErrorFirst}.
    The error estimate \eqref{bmIpError} stems from \eqref{interErrorQh}
    and \eqref{OmmSBysused}.
\end{proof}
Recall the interpolation operator $\Pih:H^1(\Omm;\bbS)\to\Sighm$ in \eqref{PihmDef}.
Combining the interpolation error estimate \eqref{interErrorPi} and the trace theorem, we obtain
\begin{equation}
\label{intermeshnormPih}
    \|\tau-\Pih\tau\|_{0,h,m}\lesssim h^s\|\tau\|_{s,\Omm}, \quad\text{for all}~\tau\in H^s(\Omm;\bbS),
\end{equation}
with $1\leq s\leq k+1$.
Although $\Pih$ does not satisfy a commutative property analogous to \eqref{defPi}, it follows from \eqref{bmldiffMo} with $l=1$ that
\begin{equation}\label{PihmProperty}
    |\bbm(\tau-\Pih\tau,\vh)|\lesssim h\|\tau-\Pih\tau\|_{0,h,m}|\vh|_{1,h,m},
\end{equation}
for all $\tau\in H^1(\Omm;\bbS)$ and $\vh\in\Vhm$.

\begin{theorem}
    Under the hypotheses of Lemma \ref{geoError} and Theorem \ref{standErrorThe}, the following error estimate holds, 
    \begin{equation}\label{L2errorEstimate}
        \|\hatsigma-\hatsigma_h\|_{0,h,m}+|\Qh\hatu-\hatuh|_{1,h,m}\lesssim h^{\min\{k+\frac{1}{2}, m+\frac{1}{2},s\}}\|u\|_{H^{s+1}(\Omega)}.
    \end{equation}
Moreover, if $\Omega$ is convex,
\begin{align}
    \label{QhuL2ErrorEsitmate}
    \|\Qh\hatu-\hatuh\|_{0,\Omm}&\lesssim h^{\min\{k+\frac{3}{2},m+1,s+1\}}\|u\|_{H^{s+1}(\Omega)},\\
    \label{uhL2RealErrorEstimate}
    \|\hatu-\hatuh\|_{0,\Omm}&\lesssim h^{\min\{k,m+1,s+1\}}\|u\|_{H^{s+1}(\Omega)}.
\end{align}
\end{theorem}
\begin{proof}
Subtracting $\bbAm(\Pih\hatsigma,\Qh\hatu;\hattauh,\hatvh)$ from both sides of \eqref{standErrorImportEq} and using the definition of $\bbAm$, we obtain
\begin{equation}\label{error88}
\begin{aligned}
    &\quad\ \bbAm(\Pih\hatsigma-\hatsigmah,\Qh\hatu-\hatuh;\hattauh,\hatvh)\\
    &=\aam(\Pih\hatsigma-\hatsigma,\hattauh)+\bbm(\Pih\hatsigma-\hatsigma,\hatvh)+\bbm(\hattauh,\Qh\hatu-\hatu)+E(\hattauh,\hatvh).
\end{aligned}
\end{equation}
The combination of \eqref{geoerr2} and \eqref{H1NormBoundm} bounds the following consistency error
\begin{equation*}
    |E(\hat{\tau},\hatv)|\leq Ch^{m+\frac{1}{2}}(\|\hattau\|_{0,h,m}+|\hatv|_{1,h,m})\|u\|_{H^{\frac{5}{2}}(\Omega)}.
\end{equation*}
This estimate, together with \eqref{error88}, the stability result of $\bbAm(\cdot,\cdot;\cdot,\cdot)$ in \eqref{Astabnorm2}, the bound on $\bbm(\hatsigma-\Pih\hatsigma,\hatvh)$ in \eqref{PihmProperty}, and the estimate of $\bbm(\hattauh,(\idI-\Qh)\hatu)$ in \eqref{bmIpError}, yields
\begin{equation}\label{PisigmaError}
\|\Pih\hatsigma-\hatsigmah\|_{0,h,m}+|\Qh\hatu-\hatuh|_{1,h,m}\lesssim \|\Pih\hatsigma-\hatsigma\|_{0,h,m} + h^q\|u\|_{H^{s+1}(\Omega)},
\end{equation}
with $q=\min\{k+\frac{1}{2},m+\frac{1}{2},s+1\}$.
Finally, applying inequality \eqref{PisigmaError} along with the interpolation error estimate of $\Pih$ in \eqref{intermeshnormPih} leads to the desired bound~\eqref{L2errorEstimate}.

The error estimate \eqref{QhuL2ErrorEsitmate} can be derived using the duality argument as follows. Consider the dual problem: find $(z,w)\in H(\Div,\Omm;\bbS)\times L^2(\Omm;\bbRn)$, such that
\begin{equation}\label{dualProblem}
\begin{cases}
    \aam(\tau,z)+\bbm(\tau,w)=0,& \text{for all}~\tau\in H(\Div,\Omm;\bbS),\\
    \bbm(z,v)=(\Qh\hatu-\hatuh, v)_{\Omm},& \text{for all}~v\in L^2(\Omm;\bbR^2).
\end{cases}
\end{equation}
The convexity of $\Omega$, together with \eqref{chap::pre::L2norm} and \eqref{geoerr2}, yields:
\begin{equation}\label{convexProperty}
 \|z\|_{1,\Omm}+\|w\|_{2,\Omm}\lesssim \|\Qh\hatu-\hatuh\|_{0,\Omm}.   
\end{equation}
Selecting test functions~$v=\Qh\hatu-\hatuh$ and $\tau=\hatsigma-\hatsigmah$ in \eqref{dualProblem},
we obtain:
\begin{equation}\label{PhuminsuhL2Equ}
  \begin{aligned}
    \|\Qh\hatu-&\hatuh\|_{0,\Omm}^2
    =\aam(\hatsigma-\hatsigmah, z-z_h) +\bbm(\hatsigma-\hatsigmah, w-w_h)\\
    &\qquad\quad+E(z_h,w_h)+\bbm(z-z_h,\Qh\hatu-\hatuh)-\bbm(z_h,\hatu-\Qh\hatu),
\end{aligned}  
\end{equation}
with $z_h=\Pih z$ and $w_h=\Qh w$.
It follows from the quasi-commutative property~\eqref{PihmProperty} and the superconvergence~\eqref{L2errorEstimate} that
\begin{equation*}
    |\bbm(z-z_h,\Qh\hatu-\hatuh)|\lesssim h^{q+1}\|z-\Pih z\|_{0,h,m}\|u\|_{H^{s+1}(\Omega)},
\end{equation*}
with $q=\min\{k+\frac{1}{2}, m+\frac{1}{2}, s\}$. Combining this with \eqref{intermeshnormPih} and \eqref{convexProperty}, we obtain
\begin{equation}
    \label{EstimateBbm4}
    |\bbm(z-z_h,\Qh\hatu-\hatuh)|\lesssim h^{\min\{k+\frac{5}{2}, m+\frac{5}{2}, s+2\}}\|\Qh\hatu-\hatuh\|_{0,\Omm}\|u\|_{H^{s+1}(\Omega)}.
\end{equation}
The definition of $z_h$ and the interpolation error estimate~\eqref{interErrorPi} imply
\begin{equation*}
    |z_h|_{1,\Thm}\leq\|z\|_{1,\Omm}+\|\Pih z-z\|_{1,\Thm}\lesssim\|z\|_{1,\Omm}.
\end{equation*}
This, together with \eqref{convexProperty}, \eqref{chap::pre::L2norm}, and the estimate of $\bbm(z_h,\hatu-\Qh\hatu)$ in \eqref{bmIpError}, yields
\begin{equation}\label{bmshuPhuFinal}
     |\bbm(z_h,\hatu-\Qh\hatu)|\lesssim h^{\min\{k+\frac{3}{2},m+2,s+2\}}\|\Qh\hatu-\hatuh\|_{0,\Omm}\|u\|_{H^{s+1}(\Omega)}.
\end{equation}
Applying the triangle inequality, \eqref{interErrorPi} and \eqref{OmmSBysused}, we derive
$
\|z_h\|_{0,\OmmS} 
\leq\|z_h-z\|_{0,\Omm}+\|z\|_{0,\OmmS}
\lesssim (1+h^{\frac{1}{2}})h^{\frac{1}{2}}\|z\|_{1,\Omm}
$. This result, combined with an analogous argument for $\|w_h\|_{0,\OmmS}$, the estimate of $|E(z_h,w_h)|$ in \eqref{geoerr2}, and \eqref{convexProperty}, leads to:
\begin{equation}\label{EshwhEstimate}
|E(z_h,w_h)|\lesssim h^{m+1}\|\Qh\hatu-\hatuh\|_{0,\Omm} \|u\|_{H^{\frac{5}{2}}(\Omega)}.
\end{equation}
Integrating the error estimates of $\hatsigma-\hatsigmah$ in \eqref{standError} and \eqref{L2errorEstimate}, we derive
\begin{align}
\label{amssh}
 |\aam(\hatsigma-\hatsigmah, z-z_h)|&\lesssim
h^{\min\{k+\frac{3}{2},m+\frac{3}{2},s+1\}}\|u\|_{s+1,\Omega}\|\Qh\hatu-\hatuh\|_{0,\Omm},\\
\label{bmwwh}
|\bbm(\hatsigma-\hatsigmah,w-w_h)|&\lesssim h^{\min\{k+2,m+\frac{5}{2},s+1\}}\|u\|_{s+1,\Omega}\|\Qh\hatu-\hatuh\|_{0,\Omm}.
\end{align}
Combining \eqref{PhuminsuhL2Equ}--\eqref{bmwwh}, we obtain the superconvergence result \eqref{QhuL2ErrorEsitmate}.

The error estimate \eqref{uhL2RealErrorEstimate} follows from the triangle inequality and the approximation property of~$\Qh$, which completes the proof.
\end{proof}

\begin{remark}
For sufficiently large $m,s$ (specifically, $m,s\geq k+1$),
\begin{equation*}
 \|\hatsigma-\hatsigmah\|_{0,\Omm}=\mathcal{O}(h^{k+\frac12}).   
\end{equation*} 
The result exhibits a convergence rate that is half an order lower than the optimal rate of $k+1$.
As shown in our analysis, this degradation arises from the error estimate~\eqref{bmIpError} introduced by boundary elements.
To remedy this issue, we introduce the following enriched finite element spaces.
First we define the enriched spaces on $\Omo$,
\begin{equation*}
\begin{aligned}
\widetilde{\Sigho}:=\{\sigmah& \in H(\Div, \Omo; \mathbb{S})~:~ \sigmah=\sigma_{c}+\sigma_{b}, \sigma_{c} \in H^{1}(\Omo ; \mathbb{S}), \\
\sigma_{c}&|_{\Ko} \in P_{k}(\Ko ; \mathbb{S}),\sigma_{b}|_{\Ko} \in \Sigma_{k, b}(\Ko), ~\text{for all}~ \Ko \in \TIho, \\
\sigma_{c}&|_{\Ko} \in P_{k+1}(\Ko ; \mathbb{S}),\sigma_{b}|_{\Ko} \in \Sigma_{k+1, b}(\Ko), ~\text{for all}~ \Ko \in \TPho\},
\end{aligned}
\end{equation*}
and
\begin{equation*}
\begin{aligned}
    \widetilde{\Vho}:=\{\vh \in L^{2}\left(\Omo; \mathbb{R}^{2}\right)~:~&\vh|_{\Ko} \in P_{k-1}(\Ko, \mathbb{R}^{2}), ~\text{for all}~\, \Ko \in \TIho, \\
    &\vh|_{\Ko} \in P_{k}(\Ko, \mathbb{R}^{2}), ~\text{for all}~ \Ko \in \TPho\},
\end{aligned}
\end{equation*}
recall that $\TIho,\TPho$ are the sets of interior and boundary triangles, respectively. 
Then the enriched spaces $\widetilde{\Sighm}$, $\widetilde{\Vhm}$ over $\Omm$ can be defined in the same way as in \eqref{HdivKm} and~\eqref{VhKm}, respectively.

Solving the discrete problem \eqref{mixpuredisplacement} with the enriched spaces $\widetilde{\Sighm}$ and $\widetilde{\Vhm}$, the convergence rate for $\sigmah$ in the $L^2$-norm can be recovered to the optimal order $k+1$:
\begin{equation*}
        \|\hatsigma-\hatsigma_h\|_{0,m}\lesssim h^{\min\{k+1, m+\frac{1}{2},s\}}\|u\|_{H^{s+1}(\Omega)}.
    \end{equation*}
Moreover, if $\Omega$ is convex,
\begin{align*}
    \|\Qh\hatu-\hatuh\|_{0,\Omm}&\lesssim h^{\min\{k+2,m+1,s+1\}}\|u\|_{H^{s+1}(\Omega)}.
\end{align*}
It is worth noting that this modification is limited to the boundary elements and incurs only a marginal increase in the number of degrees of freedom.
\end{remark}

\begin{remark}
The postprocessing of the displacement is a common practice in the analyses of mixed finite element methods: find $\hatush\in\Vhsm$ and $\hatphih\in\Vhm$, such that for all~$\Km\in\Thm$, satisfying
\begin{equation*}
    \begin{cases}
(\varepsilon(\hatush),\varepsilon(v))_{\Km}+(v,\hatphih)_{\Km}=(\maA\hatsigmah,\varepsilon(v))_{\Km},&\text{for all}~v\in\Vhsm|_{\Km},\\
(\hatush,\psi)_{\Km}=(\hatuh,\psi)_{\Km},&\text{for all}~\psi\in\Vhm|_{\Km},
    \end{cases}
\end{equation*}
where the space $\Vhsm$ is defined as
\begin{equation*}
    \Vhsm:=\{v\in L^2(\Omm;\bbRn)~:~v|_{\Km}\circ\FKm\in P_{k+1}(\Ko;\bbRn),~\text{for all}~\Ko \in \Tho\}.
\end{equation*}
Following the proof in \cite{chen2018fast} and \eqref{QhuL2ErrorEsitmate}, we can get
\begin{equation*}
    \|\hatu-\hatush\|_{0,\Omm}\lesssim h^{\min\{k+\frac{3}{2},m+1,s+1\}}\|u\|_{H^{s+1}(\Omega)}.
\end{equation*}
Moreover, after replacing $\Sighm,\Vhm$ and $\Vhsm$ with $\widetilde{\Sighm},\widetilde{\Vhm}$ and $\widetilde{\Vhsm}$, respectively, we can get
\begin{equation*}
    \|\hatu-\hatush\|_{0,\Omm}\lesssim h^{\min\{k+2,m+1,s+1\}}\|u\|_{H^{s+1}(\Omega)},
\end{equation*}
where the enriched space $\widetilde{\Vhsm}$ is defined as
\begin{equation*}
\begin{aligned}
    \widetilde{\Vhsm}:=\{v\in L^2(\Omm;\bbRn)~:~&v|_{\Km}\circ\FKm\in P_{k+1}(\Ko;\bbRn),~\text{for all}~\Ko \in \TIho,\\
    &v|_{\Km}\circ\FKm\in P_{k+2}(\Ko;\bbRn),~\text{for all}~\Ko \in \TPho\}.
\end{aligned}
\end{equation*}
\end{remark} 

\section{Numerical tests}\label{numericres}
This section presents the competitive relationship between polynomial degree~$k$ and geometric approximation order $m$ through numerical experiments on both unit disk and an asymmetric domain.
All meshes are generated by Gmsh~(version 4.13.1, http://gmsh.info, \cite{Geuzaine2009Gmsh}) with a successive uniform refinement scheme.
Errors including $\|\hatu-\hatuh\|_{L^2(\Omm)}$, $\|\hatu-\hatush\|_{L^2(\Omm)}$, $\|\hatsigma-\hatsigmah\|_{L^2(\Omega^m)}$ and~$\|\HatDiv\hatsigma-\HatDiv\hatsigmah\|_{L^2(\Omm)}$ are calculated over the approximate domain $\Omm$.

\subsection{Unit disk domain}
We test our results over the unit disk $\Omega=\{(x,y)~:~x^2+y^2\leq1\}$.
The exact solution is taken as
\begin{equation}\label{exactSolution}
u(x,y)=\big(e^{xy}\cos(x),~ e^y\sin(x+y)\big)^{T},
\end{equation}
and the Lam\'{e} constants are taken as $\lambda=1$, $\mu=1$. 

First, we perform the computation using the spaces $\Sighm$ and $\Vhm$.
Table~\ref{unitDisk} reports the convergence rates, computed as the slopes of linear regressions over the last three mesh levels in a sequence of uniformly refined meshes.
As expected, the convergence rates of stress in both the $L^2$-norm and the $H(\Div)$-norm are consistent with the theoretical prediction. 
Specifically, the convergence rate in the $L^2$-norm of~$\hatsigmah$ is $\min\{k+\tfrac{1}{2}, m+\tfrac{1}{2}\}$, which does not attain the optimal order $k+1$ even with $m > k$.
For $\HatDiv\hatsigmah$, the observed convergence rates are  $\min\{k, m+\tfrac{1}{2}\}$. When $m$ is odd, the convergence rates in the $L^2$-norm of both the displacement and its post-processed version align well with the theoretical expectation, which are $\min\{k, m+1\}$ and $\min\{k+\tfrac{3}{2}, m+1\}$, respectively.
In contrast, for even $m$, the convergence rates in the $L^2$-norm for both the approximate displacement and its post-processed solution exceed the corresponding theoretical rates by approximately half an order.

The right half of Table~\ref{unitDisk} presents results by the modified spaces~$\widetilde{\Sighm}$ and $\widetilde{\Vhm}$.
The primary improvement brought by this modification is that the convergence rates for~$\|\hatu - \hatush\|_{L^2}$ and $\|\hatsigma - \hatsigmah\|_{L^2}$ increase by half an order for~$m=k+1$, in full agreement with the theoretical prediction.
\begin{table}[htbp]
    \centering
    \caption{
    \centering
    The convergence rates for unit disk, where $N_T$ is the number of elements in the final mesh, $m$ and~$k$ are the order of meshes and finite element spaces, respectively.}
    \label{unitDisk}
	\begin{tabular}{ccc|cccc|cccc}
	\hline
    \hline
     \multirow{2}{*}{$N_T$} & \multirow{2}{*}{$k$} & \multirow{2}{*}{$m$} & \multicolumn{4}{|c|}{$\Sighm\times\Vhm$}&\multicolumn{4}{c}{$\widetilde{\Sighm}\times\widetilde{\Vhm}$}\rule{0pt}{3ex}\\
	&&& $\hatuh$ & $\hatush$ & $\hatsigmah$ & $\HatDiv\hatsigmah$& $\hatuh$ & $\hatush$ & $\hatsigmah$ & $\HatDiv\hatsigmah$ \\
     \hline
     \hline
     \multirow{4}{*}{14336} & $3$ & $1$ & 1.97              & 1.98                & 1.54                        & 1.51           & 2.05              & 2.05                & 1.58                        & 1.51                         \\ 
	 & $3$ & $2$ & 3.04              & {3.50}                & 2.50                        & 2.50   & 2.96              & {3.53}                & 2.50                        & 2.52                                   \\ 
	& $3$ & $3$ & 3.03              & {4.41}                & 3.51                        & 3.14     & 2.93              & 4.09                & 3.57                        & 2.94                                 \\ 
	& $3$ & $4$ & 3.03              & 4.49                & 3.51                        & 3.14             & 2.93              & 4.97                & 3.97                        & 2.93                          \\ \hline
    \multirow{5}{*}{14336}
    & $4$ & $1$ & 1.98              & 1.98                & 1.52                        & 1.51          & 2.05              & 2.05                & 1.58                        & 1.51                                              \\ 
	 & $4$ & $2$ & {3.50}              & {3.50}                & 2.50                        & 2.49                         & {3.54}              & {3.54}                & 2.50                        & 2.49                  \\ 
	 & $4$ & $3$ & 4.09              & 4.00                & 3.51                        & 3.53              & 3.97              & 4.08                & 3.52                        & 3.55                                     \\ 
   & $4$ & $4$ & 4.09              & {5.50}                & 4.49                        & 4.17                       & 3.94              & {5.68}                & 4.68                        & 3.89              \\ 
 & $4$ & $5$ & 4.09              & 5.49                & 4.49                        & 4.18         & 3.94              & 5.89                & 4.88                        & 3.88                                           \\ \hline
    \hline
	\end{tabular}
\end{table}

\subsection{Three-leaf domain}
As in \cite{arnold2020hellan}, we test on the $\Omega$ whose boundary is parametrized by
\begin{equation*}
\begin{cases}
    x(t)=[1+0.4\cos(3t)]\cos(t),\\
    y(t)=[1+(0.4+0.22\sin(t))\cos(3t)]\sin(t),
\end{cases}
\end{equation*}
with $0\leq t\leq2\pi$.
The coarsest mesh is shown in Figure \ref{fig:three_leaf2}.
It is evident that the domain is non-convex. To generate a regular mesh, we deliberately applied local refinement near the concave corners in advance.
The exact solution is taken the same as \eqref{exactSolution}.
\begin{figure}[htbp]
    \centering
    \subfloat[$m=1$, $N_T=332$.]{
    \label{fig:three_leaf2_1_0.500000}
    \includegraphics[scale=0.35]{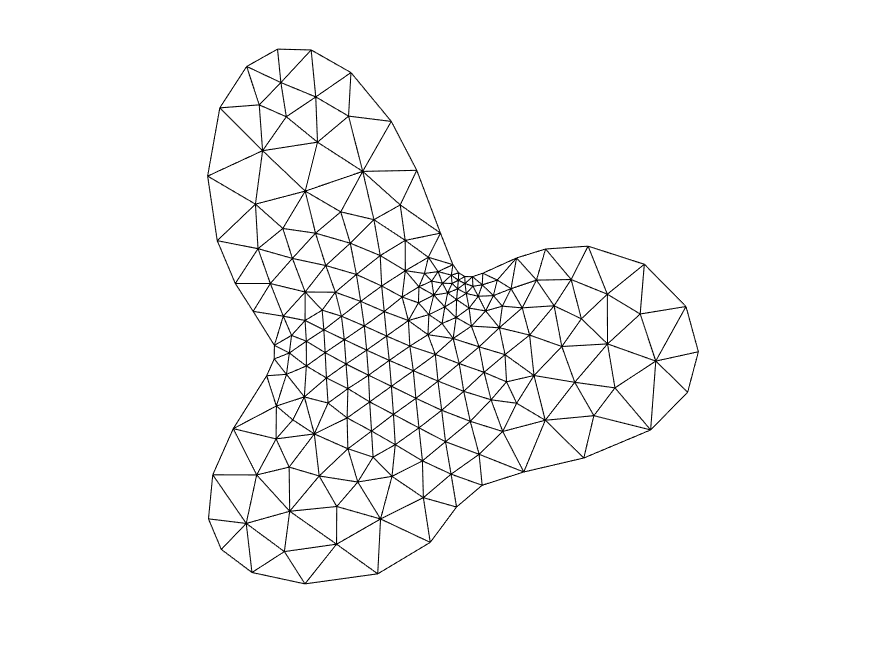}}
    \subfloat[$m=3$, $N_T=332$.]{
    \label{fig:three_leaf2_3_0.500000}
    \includegraphics[scale=0.35]{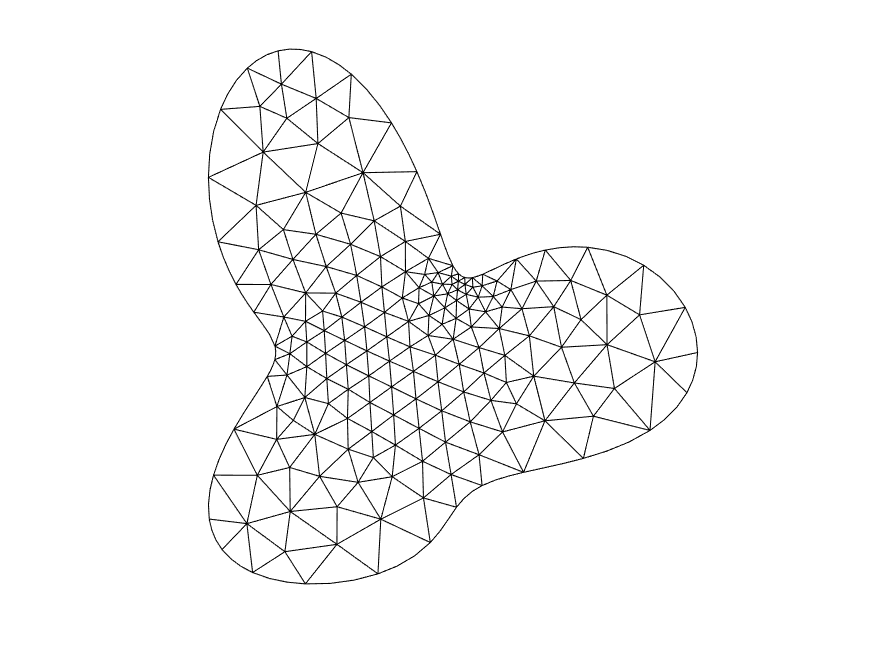}}
    \caption{Illustration of the three-leaf domain with mesh order $m=1$ (left) and $m=3$ (right).}
    \label{fig:three_leaf2}
\end{figure}
As indicated in Table~\ref{threeLeaf}, most results align with those from the first numerical example and closely match the theoretical prediction.
Minor deviations observed in a few cases may be due to the non-convexity of the domain.
Nevertheless, the method remains capable of producing reasonably accurate results even on certain non-convex domains.
\begin{table}[htbp]
    \centering
    \caption{
    \centering
    The convergence rates for three-leaf domain, where $N_T$ is the number of elements in the final mesh,~$m$ and $k$ are the order of meshes and finite element spaces, respectively.}
    \label{threeLeaf}
	\begin{tabular}{ccc|cccc|cccc}
    \hline 
    \hline
     \multirow{2}{*}{$N_T$} & \multirow{2}{*}{$k$} & \multirow{2}{*}{$m$} & \multicolumn{4}{|c|}{$\Sighm\times\Vhm$}&\multicolumn{4}{c}{$\widetilde{\Sighm}\times\widetilde{\Vhm}$}\rule{0pt}{3ex}\\
	&&& $\hatuh$ & $\hatush$ & $\hatsigmah$ & $\HatDiv\hatsigmah$& $\hatuh$ & $\hatush$ & $\hatsigmah$ & $\HatDiv\hatsigmah$ \\
    
     \hline
     \hline
     \multirow{4}{*}{84992}
      & 3 & $1$ & 1.98              & 1.98                & 1.54                        & 1.51        & 2.04              & 2.05                & 1.59                        & 1.51                                 \\
	& 3 & $2$ & 3.18              & {3.52}                & 2.50                        & 2.49      & 3.15              & {3.55}                & 2.49                        & 2.50                                \\ 
	 & 3 & $3$ & 3.13              & 4.16                & 3.53                        & {3.31}     & 2.89              & 4.08                & 3.53                        & 3.09                                   \\ 
	 & 3 & $4$ & 3.13              & 4.48                & 3.52                        & {3.32}     & 2.89              & 5.15                & 4.14                        & 2.95                                 \\ \hline
     \multirow{5}{*}{21248}
     & 4 & $1$ & 1.96              & 1.96                & 1.56                        & 1.51                 & 1.97              & 1.97                & 1.61                        & 1.51                            \\
	 & 4 & $2$ & {3.52}              & {3.52}                & 2.49                        & 2.49         & {3.54}              & {3.54}                & 2.49                        & 2.49                                 \\ 
	 & 4 & $3$ & {4.20}              & 4.02                & 3.50                        & 3.49                                    & 4.00              & 4.01                & 3.52                        & 3.47          \\ 
	 & 4 & $4$ & {4.41}              & {5.46}                & 4.48                        & {4.23}                            & 3.90              & {5.53}                & 4.48                        & {4.20}                             \\ 
	 & 4 & $5$ & {4.41}              & 5.46                & 4.47                        & {4.22}  &           3.86              & 5.99                & 5.17                        & 4.16                            \\ \hline
    \hline
	\end{tabular}
\end{table}

\section*{Conclusions}
In this article, we have extended the Hu–Zhang element to curved domains and established the discrete inf–sup conditions in both the $H(\Div,\Omega;\bbS)\times L^2(\Omega;\bbRn)$ norm and the mesh-dependent norms. 
The convergence rates of stress and displacement are demonstrated to be jointly influenced by the polynomial degree $k$ and the geometric approximation order $m$, which is further validated by two numerical experiments.
Our analysis reveals that the non-inclusion (i.e., $\Div\Sighm\not\subseteq \Vhm$) leads to the suboptimality of stress in the $L^2$-norm, which is mitigated by local $p$-enrichment on boundary elements.
For simplicity, we have restricted our discussion to the homogeneous displacement boundary condition, but the proposed analysis naturally extends to more general cases, as also indicated by our numerical results.

\section*{Acknowledgments}
The first and second authors are grateful to Professor Rui Ma from Beijing Institute of Technology for her helpful discussions and valuable suggestions in this article.

\bibliographystyle{siamplain}
\bibliography{bibfile}

\end{document}